\documentclass{amsart}
\usepackage{amsmath, amscd, amssymb, amsthm}
\usepackage{bbm}
\usepackage{latexsym}
\usepackage{amsfonts}
\usepackage{graphicx}
\usepackage[all,cmtip]{xy}
\usepackage[colorlinks,linkcolor=blue,breaklinks=blue,urlcolor=blue,citecolor=blue,anchorcolor=blue,pagebackref]{hyperref}%
\usepackage{geometry}
\newtheorem{theorem}{Theorem}
\newtheorem{lemma}{Lemma}

\newtheorem{proposition}{Proposition}

\newtheorem{question}[theorem]{Question}

\renewcommand*\backref[1]{}
\renewcommand*\backrefalt[4]{ \ifcase #1 \or (cited on page #2) \else (cited on pages #2) \fi}

\newcommand{\be}{\begin{equation}}
\newcommand{\ee}{\end{equation}}
\newcommand{\bea}{\begin{eqnarray}}
\newcommand{\eea}{\end{eqnarray}}

\newcommand{\vs}{\vspace{0.5cm}}

\def\XXint#1#2#3{{\setbox0=\hbox{$#1{#2#3}{\int}$ }
\vcenter{\hbox{$#2#3$ }}\kern-.6\wd0}}

\begin{document}

\title[Chern flat manifolds that are torsion-critical]{Chern flat manifolds that are torsion-critical}

\author{Dongmei Zhang}
\address{Dongmei Zhang. School of Mathematical Sciences, Chongqing Normal University, Chongqing 401331, China}
\email{{2250825921@qq.com}}\thanks{Zhang is supported by Chongqing graduate student research grant No.\,CYB240227. The corresponding author Zheng is partially supported by National Natural Science Foundations of China
with the grant No.\,12141101 and 12471039, Chongqing grant cstc2021ycjh-bgzxm0139, Chongqing Normal University grant 19XRC001, and is supported by the 111 Project D21024.}

\author{Fangyang Zheng}
\address{Fangyang Zheng. School of Mathematical Sciences, Chongqing Normal University, Chongqing 401331, China}
\email{20190045@cqnu.edu.cn; franciszheng@yahoo.com} \thanks{}


\begin{abstract}
In our previous work, we introduced a special type of Hermitian metrics called {\em torsion-critical,} which are non-K\"ahler critical points of the $L^2$-norm of Chern torsion over the space of all Hermitian metrics with unit volume on a compact complex manifold.  In this short note, we restrict our attention to the class of compact Chern flat manifolds, which are compact quotients of complex Lie groups equipped with compatible left-invariant metrics. Our main result states that, if a Chern flat metric is torsion-critical, then the complex Lie group must be semi-simple, and conversely, any semi-simple complex Lie group admits a compatible left-invariant  metric that is torsion-critical.

\end{abstract}

\maketitle

\tableofcontents

\markleft{Dongmei Zhang and Fangyang Zheng}

\markright{torsion-critical manifolds}

\section{Introduction and statement of result}\label{intro}

In a previous work \cite{ZhangZ}, the authors introduced {\em torsion-critical metrics} which are Hermitian metrics on a compact complex manifold $M^n$ that are critical points of the functional of $L^2$-norm of the Chern torsion tensor $T$:
\begin{equation*} \label{eq:F}
{\mathcal F} (g) = V^{\frac{1-n}{n}} \int_M |T|^2 dv, \ \ \ \ \ \ g\in {\mathcal H}_M.
\end{equation*}
Here ${\mathcal H}_M$ is the set of all Hermitian metrics on $M^n$, $V$ is the volume of $g\in {\mathcal H}_M$, and $T$ is the torsion of the Chern connection $\nabla$ defined by $\,T(x,y)=\nabla_xy-\nabla_yx-[x,y]\,$ for any vector fields $x$, $y$ on $M^n$. Since $T=0$ if and only if $g$ is K\"ahler, we see that all K\"ahler metrics are torsion-critical, and they are the absolute minimum of the functional ${\mathcal F}$. Of course we are primarily interested in critical points of ${\mathcal F}$ that are non-K\"ahler.

As proved in \cite{ZhangZ}, torsion-critical metrics are characterized by the equation:
\begin{equation} \label{eq:tor-crit}
2\sigma_1 - \sigma_2 +2(\phi + \bar{\phi}) - 2 (\xi + \bar{\xi}) = (|T|^2-\frac{n-1}{n}b)\, \omega, \ \ \ \ \ \ b=V^{-1}\int_M|T|^2dv.
\end{equation}
where $\omega$ is the K\"ahler form of $g$, and the global $(1,1)$-forms $\sigma_1$, $\sigma_2$, $\phi$ and $\xi$ are given respectively by
\begin{eqnarray*}
&& \sigma_1= \sqrt{-1}\sum_{i,j} A_{i\bar{j}} \varphi_i \wedge \overline{\varphi}_j, \ \ \ \ \ \ \sigma_2= \sqrt{-1}\sum_{i,j} B_{i\bar{j}} \varphi_i \wedge \overline{\varphi}_j, \\
&& \phi = \sqrt{-1}\sum_{i,j} \phi_{i}^j \varphi_i \wedge \overline{\varphi}_j, \ \ \ \ \ \ \ \ \xi = \sqrt{-1}\sum_{i,j} \xi_{i}^j \varphi_i \wedge \overline{\varphi}_j.
\end{eqnarray*}
Here $e$ is any local unitary frame, with dual coframe $\varphi$, and
\begin{equation} \label{eq:AB}
A_{i\bar{j}} = \sum_{r,s}T^r_{is}\overline{T^r_{js}}, \ \ \ B_{i\bar{j}} = \sum_{r,s}T^j_{rs}\overline{T^i_{rs}}, \ \ \ \phi_{i}^j = \sum_r T^j_{ir} \overline{\eta}_r, \ \ \ \ \xi_{i}^j = \sum_r T^j_{ir,\,\bar{r}} ,
\end{equation}
where $T(e_i,e_k)=\sum_j T^j_{ik}e_j$, $\eta_i=\sum_k T^k_{ki}$, and the global $(1,0)$-form $\eta =\sum_i \eta_i\varphi_i$ is called {\em Gauduchon's torsion $1$-form} which is determined  by the equation $\partial (\omega^{n-1})=-\eta \wedge \omega^{n-1}$. The index after comma stands for covariant derivatives with respect to the Chern connection $\nabla$. 

The above consideration stems from Gauduchon's celebrated work \cite{Gau84} on the functional 
\begin{equation*}
{\mathcal G} (g) = V^{\frac{1-n}{n}} \int_M |\eta|^2 dv, \ \ \ \ \ \ g\in {\mathcal H}_M.
\end{equation*}
He proved that $g$ is a critical point of ${\mathcal G}$ if and only if it satisfies the following equation:
\begin{equation} \label{eq:Gau-crit}
\sqrt{-1} (\partial \overline{\eta} - \overline{\partial} \eta - \eta \wedge \overline{\eta}) = a\, \omega
\end{equation}
for some constant $a \geq 0$. He also proved \cite[Theorem III.4]{Gau84} that, when $n=2$, any critical point of ${\mathcal G}$ must be {\em balanced,} namely, with $\eta=0$. In \cite[Proposition 1]{ZhangZ}, we generalized it to higher dimensions, namely, we proved that for any $n$, $g$ is a critical point of ${\mathcal G}$ if and only if it is balanced. In other words, the functional ${\mathcal G}$ does not have any non-trivial critical points. 

Since $T$ and $\eta$ carry the same amount of information when $n=2$, we know that there is no non-K\"ahler torsion-critical metric when $n=2$. But when $n\geq 3$, there do exist non-K\"ahler metrics that are torsion-critical. In \cite{ZhangZ}, we gave an explicit example of such kind, which establishes the non-emptiness of the set (of non-K\"ahler torsion-critical metrics). The example is a Chern flat threefold. In the same paper, we also proved a non-existence result (\cite[Proposition 4]{ZhangZ}) which shows that such metrics are relatively rare. 

The main purpose of this short note is to answer a question raised at the end of \S 1  of \cite{ZhangZ}, which says that: {\em Classify all torsion-critical metrics that are Chern flat.}

By the classic work of Boothby \cite{Boothby}, compact Chern flat manifolds are exactly compact quotients of complex Lie groups (equipped with  compatible left-invariant metrics). We have the following

\begin{proposition}  \label{prop1}
Let $(M^n,g)$ be a compact Chern flat manifold. Denote by $G$ its universal cover which is a complex Lie group. If $g$ is non-K\"ahler torsion-critical, then $G$ must be semi-simple. Conversely, given any semi-simple complex Lie group $G$, it admits a compatible left-invariant metric $g_0$ satisfying (\ref{eq:tor-crit}), so any compact quotient of $(G,g_0)$ is non-K\"ahler torsion-critical.  
\end{proposition}

The metric $g_0$ on $G$ is called the {\em canonical metric}. In \cite{PodestaZ} an explicit computation was given. See also \cite{Gordon, Helgason, Lafuente} for earlier literatures on this construction. At this point, we do not know if any torsion-critical metric $g$ on a semi-simple complex Lie group $G$ must be a constant multiple of $g_0$ or not, although it is true in some special cases (such as $\mbox{SL}(2, {\mathbb C})$).

\vspace{0.3cm}

\section{Proof of Proposition \ref{prop1}} 

Recall that a Hermitian manifold $(M^n,g)$ is said to be {\em Gauduchon} if $\partial \overline{\partial} (\omega^{n-1})=0$. Equivalently, this is the case if and only if $|\eta|^2 = \chi$, where $\chi $ is the global function on $M$ defined by $\chi = \sum_i 
\eta_{i,\bar{i}}$ under any local unitary frame $e$, again the index after comma stands for covariant derivatives with respect to the Chern connection $\nabla$. 
A result of Angella, Istrati, Otiman, and Tardini, \cite[Proposition 17]{AIOT} says the following:

\begin{proposition}[\cite{AIOT}] \label{propAIOT}
Denote by $[g]$ the conformal class of $g$, namely, the set of all Hermitian metrics on $M^n$ conformal to $g$. Then $g$ is a critical point of ${\mathcal F}|_{[g]}$ if and only if
\begin{equation}\label{eq:trace}
4\big(|\eta|^2-\chi \big) = (n-1) \big(|T|^2-b\big),
\end{equation}
where $b$ is the average value of $|T|^2$. In particular, such a metric will be Gauduchon if and only if $|T|^2$ is a constant.
\end{proposition}

Now let $(M^n,g)$ be a compact Chern flat manifold. By \cite{Boothby}, we know that its universal cover is $(G,\tilde{g})$, where $G$ is a complex Lie group and $\tilde{g}$ is a left-invariant metric on $G$ compatible with the complex structure of $G$. It is known (\cite[Theorem 3]{YZ}) that $g$ must be balanced, namely, $\eta =0$. (Note that the balanced condition is a natural generalization of K\"ahlerness, and is widely studied in non-K\"ahler geometry. See \cite{Fu} and references therein for more discussion. Besides Chern flat manifolds, Bismut flat manifolds were classified in \cite{WYZ}, while the classification of Riemannian flat Hermitian manifolds are only known in complex dimension $3$ so far \cite{KYZ}).  So in particular, $\phi =0$. By the first identity in Lemma 7 of \cite{YZ}, we know that under any local unitary frame $e$ we have
$$ T^{\ell}_{ik,\bar{j}} = R_{k\bar{j}i\bar{\ell}} - R_{i\bar{j}k\bar{\ell}}, \ \ \ \ \ \forall \ 1\leq i,j,k,\ell \leq n. $$
Here $T^j_{ik}$, $R_{i\bar{j}k\bar{\ell}}$  are respectively the components under $e$ of the torsion and curvature of the Chern connection $\nabla$, and index after comma stands for covariant derivatives with respect to $\nabla$. From this we know that any Chern flat metric $g$  (namely, $R=0$) will have $\xi=0$. Hence by the torsion-critical equation (\ref{eq:tor-crit}) we obtain the following
\begin{lemma} \label{lemma1}
Let $(M^n,g)$ be a compact Chern flat manifold. Then it is torsion-critical if and only if 
\begin{equation} \label{eq:6}
2A - B = \frac{b}{n}I,
\end{equation} 
for some constant $b$. 
\end{lemma} 
By taking trace on the above equations, we see that the constant $b$ above is necessarily the average value of the norm square $|T|^2$ of the Chern torsion, as we have seen in ({\ref{eq:tor-crit}). Since compact Chern flat metrics and always balanced (\cite[Theorem 3]{YZ}), hence Gauduchon, so by Proposition \ref{propAIOT} we know that $|T|^2$ is a constant function which equals to $b$. The following algebraic fact is well-known, so we omit its proof:
\begin{lemma} \label{lemma2}
A Lie algebra is semi-simple if and only if it does not contain any non-trivial abelian ideal. 
\end{lemma}


Now we are ready to prove the first part of Proposition \ref{prop1}:
\begin{lemma}  \label{lemma3}
Let $(M^n,g)$ be a compact Chern flat manifold. Denote by $G$ its universal cover which is a complex Lie group. If $g$ is non-K\"ahler torsion-critical, then $G$ must be semi-simple. 
\end{lemma}

\begin{proof}
Assume the contrary that $G$ is not semi-simple. Then by Lemma \ref{lemma2} its Lie algebra ${\mathfrak g}$ will contain a non-trivial abelian ideal ${\mathfrak a}$. Let $\{ e_1, \ldots , e_n\}$ be a unitary basis of ${\mathfrak g}$, so that $\{ e_1, \ldots , e_r\}$ is a basis of ${\mathfrak a}$. Since ${\mathfrak a}$ is an abelian ideal, and $[e_a,e_b]= - \sum_c T^c_{ab}e_c$, we have 
$$  T^{\ast }_{\alpha \beta}=T^i_{\alpha j} =0, \ \ \ \ \forall \ 1\leq \alpha, \beta \leq r, \ \forall \ r+1\leq  i,j\leq n, \ \forall \ 1\leq \ast \leq n.$$
By the definition of $A$ and $B$, we have
$$ A_{\alpha \bar{\alpha}} = \sum_{a,b=1}^n |T^a_{\alpha b}|^2 =  \sum_{\beta =1}^r \sum_{i=r+1}^n |T^{\beta}_{\alpha i}|^2, \ \ \  B_{\alpha \bar{\alpha}} = \sum_{a,b=1}^n |T^{\alpha}_{a b}|^2 =  2\sum_{\beta =1}^r \sum_{i=r+1}^n |T^{\alpha}_{\beta i}|^2 + \sum_{i,j=r+1}^n |T^{\alpha }_{ij}|^2. $$
Therefore
$$ \sum_{\alpha =1}^r \big( 2A_{\alpha \bar{\alpha}}- B_{\alpha \bar{\alpha}}\big) = - \sum_{i,j,\alpha} |T^{\alpha }_{ij}|^2 \leq 0.$$
This of course would contradict with equation (\ref{eq:6}) in Lemma \ref{lemma1}, as $b=|T|^2>0$. So we know that if a compact Chern flat manifold $(M^n,g)$ is non-K\"ahler torsion-critical, then the complex Lie group $G$ which is the universal cover of $M$ can not contain any abelian ideal, hence $G$ is semi-simple. This completes the proof of the lemma.
\end{proof}

Now we are ready to prove Proposition \ref{prop1}.

\begin{proof}[{\bf Proof of Proposition \ref{prop1}.}]
We just need to show that any semi-simple complex Lie group $G$ admits a left-invariant metric compatible with its complex structure that is torsion-critical. Let $G$ be a  semi-simple complex Lie group. We want to show that its {\em canonical metric} $g_0$ on $G$ is torsion-critical. First let us  recall the definition of $g_0$ from \cite{PodestaZ}. The construction is well-known and we refer the readers to \cite{Gordon, Helgason, Lafuente} for more references. Here we will follow the notations and calculation in \cite{PodestaZ}.

Let ${\mathfrak g}$ be the Lie algebra of $G$, and we  denote by $\mathfrak g_{\mathbb R}$ the Lie algebra $\mathfrak g$ considered as a real algebra. A Cartan decomposition of  $\mathfrak g_{\mathbb R}$ is given by ${\mathfrak g_{\mathbb R}} = {\mathfrak u} + J{\mathfrak u}$, where ${\mathfrak u}$ is the Lie algebra of a compact semi-simple Lie group $U$. Let us denote by $K$ the {\em Cartan-Killing form} of ${\mathfrak g_{\mathbb R}}$, namely, 
$$ K(x,y) = \mbox{tr}(\mbox{ad}_x\mbox{ad}_y), \ \ \ \forall \ x,y \in {\mathfrak g_{\mathbb R}}. $$
The canonical metric $g_0$ on ${\mathfrak g_{\mathbb R}}$ is defined by
$$g_0|_{{\mathfrak u} \times {\mathfrak u}} = -K,\ \ \ g_0({\mathfrak u},J{\mathfrak u})=0,\ \ \ g_0|_{J{\mathfrak u} \times J{\mathfrak u}} = K.$$
Note that this metric depends on the choice of a compact real form ${\mathfrak u}$, but as two such real forms are conjugate by an inner automorphism of $\mathfrak g_{\mathbb R}$, the corresponding metrics are holomorphically isometric. 

Since $J$ is an integrable complex structure on ${\mathfrak g_{\mathbb R}}$, we have by definition that
$$ [x,y]-[Jx,Jy] + J[Jx,y] + J[x,Jy] = 0, \ \ \ \ \forall \ x,y \in {\mathfrak g_{\mathbb R}}. $$
Also,  ${\mathfrak g}$ being a complex Lie algebra means that 
$$ [x,y] + [Jx,Jy] = 0,  \ \ \ \ \forall \ x,y \in {\mathfrak g_{\mathbb R}}. $$
Combine the two we get
\begin{equation*}
[Jx,y]=J[x,y], \ \ \ \ \forall \ x,y \in {\mathfrak g}_{\mathbb R}. 
\end{equation*}
Now let $\{ u_1, \ldots , u_n, Ju_1, \ldots , Ju_n\}$ be an orthonormal basis of ${\mathfrak g}_{\mathbb R} = {\mathfrak u}+ J{\mathfrak u}$, and for convenience write $Ju_i=u_{i^{\ast}}=u_{n+i}$, $1\leq i\leq n$. 
Denote by $S^j_{ik}$ the structure constants of ${\mathfrak u}$, we have
$$ [u_i,u_k] =\sum_{j=1}^n S^j_{ik} u_j = - [u_{i^{\ast}}, u_{k^{\ast}}], \ \ \ [u_i, u_{k^{\ast}}] = \sum_{j=1}^n S^j_{ik}u_{j^{\ast}}. $$
For any $1\leq i,j,k \leq n$, we have
$$ \mbox{ad}_{u_i}\mbox{ad}_{u_k}(u_j) = [u_i, [u_k, u_j]] = \sum_{r,\ell =1}^n  S^{\ell}_{ir} S^r_{kj} u_{\ell}, \ \ \ \ \ \mbox{ad}_{u_i} \mbox{ad}_{u_k}(u_{j^{\ast}}) = \sum_{r,\ell =1}^n  S^{\ell}_{ir} S^r_{kj} u_{\ell^{\ast}}. $$
Therefore we have 
\begin{equation} \label{eq:SS}
\delta_{ik}=g_0(u_i,u_k) = - \mbox{tr}(\mbox{ad}_{u_i}\mbox{ad}_{u_k}) = -2 \sum_{j,r=1}^n  S^{j}_{ir} S^r_{kj}, \
 \ \ \ \forall \ 1\leq i,k\leq n.
\end{equation}

Next let us write $e_i=\frac{1}{\sqrt{2}}(u_i - \sqrt{-1}Ju_i)$, $1\leq i\leq n$. Then $\{ e_1, \ldots , e_n\}$ becomes a unitary basis for $({\mathfrak g}, g_0)$. Under the basis $e$, we compute
\begin{equation*}
[e_i,e_k] \,= \,\frac{1}{2}[ u_i - \sqrt{-1}Ju_i, \,u_k - \sqrt{-1}Ju_k] \, = \, [u_i,u_k] - \sqrt{-1}J [u_i,u_k] \, = \, \sqrt{2}\sum_{j=1}^n S^j_{ik}e_j.
\end{equation*}
Thus the Chern torsion components under $e$ becomes 
\begin{equation} \label{eq:TS}
T^j_{ik} = - \sqrt{2}S^j_{ik}, \ \ \ \ \ 1 \leq i,j,k\leq n. 
\end{equation}
It has been proved in \cite[Theorem 1.2]{PodestaZ} that $g_0$ is Bismut torsion parallel, namely, $\nabla^bT^b=0$ where $\nabla^b$ is the Bismut connection and $T^b$ is its torsion (see \cite{Bismut}. Since $T^b$ can be expressed in the Chern torsion $T^c=T$, this condition is equivalent to $\nabla^bT=0$. Bismut torsion parallel manifolds, and in particular those Hermitian manifolds where the Bismut curvature obeys all K\"ahler symmetries, were actively studied in recent years. See \cite{YZZ, ZZ, Zheng} and the references therein for more discussions). Hence $\nabla^bC=0$ where the symmetric $2$-tensor $C$ is defined by $C_{ij}=\sum_{r,s=1}^n T^r_{is}T^s_{jr}$ under any unitary frame. Under our particular choice of frame $e$ above, we have by (\ref{eq:SS}) and (\ref{eq:TS}) that
$$ C_{ij} = \sum_{r,s=1}^n T^r_{is}T^s_{jr} = 2 \sum_{r,s=1}^n S^r_{is}S^s_{jr} = -\delta_{ij}, \ \ \ \ \forall \ 1\leq i,j\leq n. $$
On the other hand, since $g_0$ is Chern flat, the Bismut connection $\nabla^b$ has coefficients (see for instance \cite{YZ} or \cite{VYZ})
$$ \nabla^b_{e_k}e_i = \sum_{j=1}^n T^j_{ik}e_j. $$
So the fact that $\nabla^bC=0$ yields
$$ \sum_{r=1}^n \big( T^r_{ik}C_{rj} + T^r_{jk}C_{ir}\big) = 0, \ \ \ \ \forall  \ 1\leq i,j,k\leq n. $$
That is, 
\begin{equation}
 T^j_{ik}  + T^i_{jk} =0 , \ \ \ \ \ \ \ \forall \ 1\leq i,j,k\leq n.
\end{equation}
Since $T$ is skew-symmetric with respect to its lower two indices, we also have $T^j_{ki}+T^i_{kj}=0$. By the definition (\ref{eq:AB}) for tensors $A$ and $B$, we have
$$ A_{i\bar{j}} =\sum_{r,s=1}^n T^r_{is} \overline{T^r_{js}} = - \sum_{r,s=1}^n T^s_{ir} \overline{T^r_{js}} = -2 \sum_{r,s=1}^n S^s_{ir} S^r_{js} =\delta_{ij}. $$
Similarly,
$$ B_{i\bar{j}} =\sum_{r,s=1}^n T^j_{rs} \overline{T^i_{rs}} = - \sum_{r,s=1}^n T^s_{jr} \overline{T^r_{is}} = -2 \sum_{r,s=1}^n S^s_{jr} S^r_{is} =\delta_{ij}. $$
So $A=B=I$ and by Lemma \ref{lemma1} we know that the canonical metric  $g_0$ on the semi-simple complex Lie group $G$ is torsion-critical. This completes the proof of Proposition \ref{prop1}. 
\end{proof}

\begin{question}
Given a semi-simple complex Lie group $G$, let $g$ be any compatible left-invariant metric. If $g$ is torsion-critical, must $g$ be a constant multiple of the canonical metric $g_0$? Or on a weaker term, must $g$ be a product metric with respect to the simple factors of $G$?
\end{question}

We do not know the answer to either of these questions. Here we remark that, in the special case when $G=\mbox{SL}(2,{\mathbb C})$, the answer is yes. To see this, let ${\mathfrak g}$ be the Lie algebra of $G$, and $\{ e_1, e_2, e_3\}$ be a basis of ${\mathfrak g}$, given by traceless matrices
$$ e_1= \frac{1}{2} \left[ \begin{array}{cc} 0 & 1 \\ -1 & 0 \end{array} \right] , \ \ \ \ e_2= \frac{1}{2} \left[ \begin{array}{cc} 0 & i \\ i & 0 \end{array} \right] , \ \ \ \ e_3= \frac{1}{2} \left[ \begin{array}{cc} i & 0 \\ 0 & -i \end{array} \right] , $$
then we have 
$$ [e_1, e_2]=e_3, \ \ \ \ [e_2, e_3]=e_1,\ \ \ \ [e_3, e_1]=e_2. $$
The canonical metric $g_0$ is the one using $e$ as a unitary frame. Under $e$, its only non-vanishing Chern torsion components are $$T^1_{23}=T^2_{31}=T^3_{12} =-1. $$
Now suppose $g$ is another metric on ${\mathfrak g}$. Then by a unitary change of $e$ if necessary, we may assume that there are positive constants $a_i$, $1\leq i\leq 3$, so that $\tilde{e}$ becomes a unitary frame of $g$ where $\tilde{e}_i=a_ie_i$. Since
$$ [\tilde{e}_i, \tilde{e}_j ] = - \sum_k \big( \frac{a_ia_j}{a_k}T^k_{ij} \big) \tilde{e}_k = - \sum_k \tilde{T}^k_{ij}\, \tilde{e}_k , $$
we see that the components of the $A$, $B$ tensors for $g$ under the frame $\tilde{e}$ becomes
$$ \tilde{A}_{i\bar{j}} = \sum_{r,s} \tilde{T}^r_{is}  \overline{ \tilde{T}^r_{js} } = a_ia_j \sum_{r,s} \frac{a_s^2}{a_r^2} T^r_{is}  \overline{ T^r_{js}}= a_ia_j\delta_{ij} \sum_{r,s}\frac{a_s^2}{a_r^2} |T^r_{is}|^2 .$$
So $\tilde{A}$ is diagonal and 
$$ \tilde{A}_{i\bar{i}} = a_i^2 \big( \frac{a_j^2}{a_k^2} + \frac{a_k^2}{a_j^2} \big) , $$
where $(ijk)$ is a cyclic permutation of $(123)$. Similarly,
$$ \tilde{B}_{i\bar{j}} = \sum_{r,s} \tilde{T}^j_{rs}  \overline{ \tilde{T}^i_{rs} } = \frac{1}{a_ia_j} \sum_{r,s} a_s^2a_r^2 T^j_{rs}  \overline{ T^i_{rs}}= \frac{\delta_{ij}}{a_i^2} \sum_{r,s}a_s^2a_r^2 |T^i_{rs}|^2 .$$
Hence $\tilde{B}$ is also diagonal with 
$$ \tilde{B}_{i\bar{i}} = \frac{1}{a_i^2}2 a_j^2a_k^2, $$
where $(ijk)$ is a cyclic permutation of $(123)$. If $g$ is torsion-critical, then $2\tilde{A}_{i\bar{i}}- \tilde{B}_{i\bar{i}} = b$ for each $i$, or equivalently,
$$ a_i^4a_j^4 + a_i^4a_k^4 - a_j^4 a_k^4 = \frac{b}{2}a_i^2a_j^2a_k^2  $$
for all $(ijk)$ that is the cyclic permutation of $(123)$. Since each $a_i$ is assumed to be positive, the above implies that $a_1=a_2=a_3$. This shows that $g$ is a constant multiple of $g_0$.

\vspace{0.3cm}

\vs

\noindent\textbf{Acknowledgments.} The second named author would like to thank Bo Yang and Quanting Zhao for their interests and/or helpful discussions. He would also like to thank Fabio Podest\`a from whom he learned useful relevant properties on Lie algebras. Both authors are very grateful to the anonymous referee who made a number of  corrections/suggestions  which enhanced the readability of the article.

\vs


\begin{thebibliography}{99}



\bibitem  {AIOT}  D. Angella, N. Istrati, A. Otiman, N. Tardini, \emph{Variational Problems in Conformal Geometry,}
J. Geom. Anal. {\bf  31} (2021), 3230-3251.




\bibitem {Bismut} J.-M. Bismut, \emph{A local index theorem for non-K\"ahler manifolds,}  Math. Ann. {\bf 284} (1989), no.\,4, 681-699.

\bibitem {Boothby} W. Boothby, \emph{Hermitian manifolds with zero curvature,} Michigan Math. J. {\bf 5} (1958), no.\,2, 229-233.



\bibitem {Fu} J-X Fu,  {\em On non-K\"ahler Calabi-Yau threefolds with balanced metrics.} Proceedings of the International
Congress of Mathematicians. Volume II, 705-716, Hindustan Book Agency, New Delhi, 2010.

\bibitem {Gau84} P. Gauduchon, \emph{La $1$-forme de torsion d'une vari\'et\'e hermitienne compacte.} Math. Ann. {\bf 267} (1984), no.\,4, 495-518.



\bibitem{Gordon} C. S. Gordon, \emph{Naturally reductive homogeneous Riemannian manifolds,} Can. J. Math. \textbf{37} (1985), 467-487.


\bibitem{Helgason} S. Helgason, \emph{Differential Geometry, Lie groups and symmetric spaces,} Academic Press, Inc. (1978).



\bibitem {KYZ} G. Khan, B. Yang, and F. Zheng, \emph{The set of all orthogonal complex strutures on the flat $6$-torus,} Adv. Math. {\bf 319} (2017), 451-471.


\bibitem{Lafuente} R. A. Lafuente, M. Pujia, L. Vezzoni, \emph{Hermitian curvature flow on unimodular lie groups and static invariant metrics,}
Trans. Amer. Math. Soc., \textbf{373} (2020), 3967-3993.


\bibitem{PodestaZ} F. Podest\`a and F. Zheng, \emph{A note on compact homogeneous manifolds with Bismut parallel torsion,} arXiv: 2310.14002




\bibitem {VYZ} L. Vezzoni, B. Yang, and F. Zheng, \emph{ Lie groups with flat Gauduchon connections,} Math. Zeit. \textbf{293} (2019), Issue 1-2, 597-608.


\bibitem {WYZ} Q. Wang, B. Yang, and F. Zheng, \emph{On Bismut flat manifolds,} Trans. Amer.Math.Soc., {\bf 373} (2020), 5747-5772.

\bibitem {YZ} B. Yang and F. Zheng, \emph{On curvature tensors of Hermitian manifolds,} Comm. Anal. Geom. {\bf 26} (2018), no.\,5, 1193-1220.



\bibitem  {YZZ} S.-T. Yau, Q. Zhao, and F. Zheng, \emph{On Strominger K\"ahler-like manifolds with degenerate torsion,} Trans. Amer. Math. Soc. \textbf{376} (2023), no.\,5, 3063-3085.


\bibitem {ZhangZ} D. Zhang and F. Zheng, \emph{On a variational theorem of Gauduchon and torsion-critical manifolds,} Proc. Amer. Math. Soc. {\bf 151} (2023), no.\,4, 1749-1762.

\bibitem {ZZ} Q. Zhao and F. Zheng, \emph{Strominger connection and pluriclosed metrics,} J. Reine Angew. Math. (Crelles), \textbf{796} (2023), 245-267. 


\bibitem {Zheng} F. Zheng, \emph{Some recent progress in non-K\"ahler geometry,} Sci. China Math., {\bf 62} (2019), no.\,11, 2423-2434.



\end{thebibliography}
\end{document}